\documentclass[12pt]{article}
\usepackage{amssymb}
\usepackage[latin1]{inputenc}
\usepackage[T1]{fontenc}
\usepackage{amsthm}
\usepackage{amsmath}
\usepackage{verbatim}

\usepackage[dvipsnames, usenames]{color}
\newtheorem{thm}{Theorem}
\theoremstyle{definition}
\newtheorem{defn}{Definition}
\newtheorem{lem}{Lemma}

\begin{document}

\date{}


\title{\textcolor{black}{On  the  absolute irreducibility of hyperplane sections of generalized Fermat varieties in $\Bbb{P}^3$ and 
the conjecture on exceptional APN functions: the Kasami-Welch degree case}}

\author{Moises  Delgado\thanks{Department of Mathematics,
University of Puerto Rico, \textcolor{black}{Cayey} Campus, San Juan PR,
 USA. moises.delgado@uprrp.edu.}\\
Heeralal  Janwa \thanks{Department of Mathematics, University of
Puerto Rico, Rio Piedras Campus, San Juan PR, USA.
hjanwa@uprrp.edu}} \maketitle









\begin{abstract}
Let $f$ be a function on a finite field $F$. The decomposition of the
generalized Fermat variety $X$ defined by
the  multivariate polynomial of degree $n$,
$\phi(x,y,z)=f(x)+f(y)+f(z)$ in $\Bbb{P}^3(\overline{\mathbb{F}}_2)$,
 plays a crucial  role in the study of almost perfect non-linear (APN) functions
and exceptional APN functions. Their structure depends fundamentally on the Fermat varieties corresponding to the  monomial  functions of exceptional degrees $n=2^k+1$ and
$n=2^{2k}-2^k+1$ (Gold and Kasami-Welch numbers, respectively).   Very important
results for these have been obtained by Janwa, McGuire and Wilson in
\cite{Janwa:McGuire:Wilson:AJA:95,Janwa:Wilson:LNCS:93}.
In this paper we study $X$ related to the Kasami-Welch degree
monomials and its decomposition into
absolutely irreducible components. We show that, in this
decomposition, the components
intersect transversally at a singular point.

 This structural fact implies that the corresponding generalized
Fermat hypersurfaces, related to Kasami-Welch degree polynomial
families, are absolutely irreducible. In particular, we prove that if $f(x)=x^{2^{2k}-2^k+1}+h(x)$,
where $\deg(h)\equiv 3{\pmod 4}$, then the corresponding APN multivariate hypersurface is absolutely irreducible, and hence $f(x)$ is not exceptional APN function. We also prove conditional result in the case when ${\rm  deg}(h)\equiv 5{\pmod 8}$. Since for odd degree $f(x)$, the conjecture needs to be resolved only for the Gold degree and the Kasami-Welch degree cases our results contribute substantially
to the proof of the  conjecture on exceptional APN functions---in the
hardest case: the Kasami-Welch degree. 
\end{abstract}



\section{Introduction}
\textcolor{black}{
\textcolor{black}{An almost perfect nonlinear (APN) function  (necessarily  a polynomial function)  on a finite field   $\mathbb{F}$ is a non-linear
 function  that is  very useful in cryptography because of its excellent
resistance to differential cryptanalysis as was demonstrated by  Nyberg \textcolor{black}{and Knudsen} \cite{Nyberg:Knudsen:LNCS:92}.  APN functions can be related to  a host of  problems in coding theory, sequence design, exponential sums,
projective geometry,  block designs, \textcolor{black}{and} permutation polynomials.  Therefore  APN  functions were well studied even before Nyberg defined them, and  by now  are well known objects of research. 
 In this article, we make substantial progress towards \textcolor{black}{the} resolution of the main conjecture (stated at the end of this section) on the existence of
exceptional APN functions  on finite fields.   We also contribute to \textcolor{black}{the} understanding  of the mysteries surrounding this conjecture.}
}


\textcolor{black}{
Up until now, the main tool used by most  researchers in the study of  exceptional APN functions,  has  been the method of Janwa, McGuire and Wilson in  \cite{Janwa:McGuire:Wilson:AJA:95} to prove the absolute irreducibility of multivariate polynomials.  The algorithmic approach in    \cite{Janwa:McGuire:Wilson:AJA:95}  is  based on  intersection multiplicity theory and Bezout's theorem, and computations initiated by Janwa and Wilson in  \cite{Janwa:Wilson:LNCS:93}.
}

\textcolor{black}{
Delgado and Janwa, in \cite{Delgado:Janwa:ArXiv:12,Delgado:Janwa:ArXiv:16,Delgado:Janwa:DCC:16}, based their
techniques of proving absolute irreducibility on repeated hyperplane intersections, linear
transformations, reductions, and properties of known  APN monomial functions.
These \textcolor{black}{allowed} us to overcome the very difficult  multiplicity
computations for hypersurfaces in the projective space $\mathbb{P}^3$.
As a consequence, \textcolor{black}{we established} absolute irreducibility of a class of
\textcolor{black}{Gold degree} multivariate polynomials over finite fields. \textcolor{black}{In this article, our methods for proving
absolute irreducibility are based on the decompositions of symmetric varieties as well as the transversal intersections of its components. This methods permit us to prove the absolute irreducibility of new class of Kasami-Welch degree multivariate polynomials, which implies a contribution on perhaps the hardest case of the Exceptional APN conjecture, as stated in the abstract}
}


The first step in many applications of algebraic geometry to coding theory,  cryptography,  number theory and other disciplines,  is the demonstration of absolute irreducibility of a given  variety.  Indeed, absolute irreducibility is a necessary condition in the applications of  the bounds of Weil,  Bombieri, Deligne, Lang-Weil,  Ghorpade-Lachaud, and others that estimate the number  of rational points on the corresponding varieties, or give bounds on exponential sums along curves.
Except for the well known  Eisenstein criterion (applicable in very restrictive cases), only few scattered results were known for proving absolute irreducibility (see Schmidt \cite{Schmidt:KP:04}---mostly applicable to Kummer and Artin-Schreier type of  extensions),
before the major breakthrough in  \cite{Janwa:McGuire:Wilson:AJA:95}.

 \textcolor{black}{
Therefore, our  techniques \textcolor{black}{(in this paper and in \cite{Delgado:Janwa:ArXiv:12,Delgado:Janwa:ArXiv:16,Delgado:Janwa:DCC:16, Delgado:Janwa:AMC:16})} and results are of independent interest.
In essence, our techniques are precursor to what would be multiplicity analysis in higher dimensions. These techniques could be used to prove absolute irreducibility of other multivariate polynomials
(for  possible applications, see Lidl and Niederreiter \cite{Lidl:Niederreiter:C:97}).}

\textcolor{black}{An important fact that we should remark is that our results are far better than other previous results on Gold and Kasami-Welch degree multivariate polynomials. The results of Aubry, McGuire, Rodier, Ferard and Oyono, in \cite{Aubry:McGuire:Rodier:CM:10,Rodier:Hal:11,Ferard:Oyono:CM:12}, provides absolutely irreducible families of polynomials with a big gap between the two higher degree terms, ruling out a considerable number of members of the families. Delgado and Janwa in \cite{Delgado:Janwa:DCC:16} and in this article (see Theorems \ref{thm:3mod4:Delgado-Janwa:JA:16} and    \ref{thm:5mod8:Delgado-Janwa:JA:16}) overcame this obstacle and establish absolute irreducibility for almost all Gold and Kasami-Welch degree polynomials of the form $x^n+h(x)$ where degree of $h$ is an odd number. Thus, our results contribute significantly to the proof of the conjecture.
}

\textcolor{black}{Some of the generalized Fermat  hypersurfaces are quite interesting. The  monomial hypersurfaces correspond to cyclic codes. For example, the   monomial $x^7$ leads to the Klein Quartic, whose zeta function has been computed, and applied to determine the number of codewords of weight four in the corresponding cyclic codes in  \cite{Janwa:Wilson:IEEE:16}. Similarly, the other   hypersurfaces can be used to analyze weight distribution of the corresponding  codes.   Some of the absolutely irreducible multivariate polynomials could also  be used  for the construction of algebraic geometric codes from the corresponding curves and  surfaces.
}

\begin{defn}\label{def1}\cite{Nyberg:Knudsen:LNCS:92}
\textcolor{black}{Let} $L=\mathbb{F}_q$, with $q=p^n$ for
positive \textcolor{black}{integer} $n$. A function \\  $f:L\rightarrow L$ is said to be
\textbf{almost perfect nonlinear }(APN) on $L$ if for all $a,b \in
L$, $a \neq 0$, the following equation
\begin{equation}\label{eqn:intro-1}
f(x+a)-f(x)=b
\end{equation}
has at most 2 solutions.
\end{defn}

Equivalently, for $p=2$, $f$ \textcolor{black} is APN if \textcolor{black}{the cardinality of}
the set $\{f(x+a)\textcolor{black}{-}f(x):x \in L \}$ \textcolor{black}{is} at least
$2^{n-1}$ for each $a\in L^{\ast}$. The best
known examples of APN functions are the Gold
\textcolor{black}{function} $f(x)=x^{2^k+1}$\textcolor{black}{,}  and the
Kasami-Welch
function \textcolor{black} {$f(x)=x^{2^{2k}-2^k+1}$;} they  are APN on any field
$\mathbb{F}_{2^n}$\textcolor{black}{} when $k$
\textcolor{black}{and} $n$ are relatively
\textcolor{black}{prime}. The \textcolor{black}{Welch} function $f(x)=x^{2^r+3}$
\textcolor{black}{is also APN on $\mathbb{F}_{2^n}$ when $n=2r+1$ (see \cite{Janwa:Wilson:LNCS:93} and \cite{Janwa:McGuire:Wilson:AJA:95}) .}

The APN property is invariant under  transformations of
functions.

A function $f:L \rightarrow L$    is linear if and only if $f$ is a
linearized polynomial over $L$, that is,
$$\textcolor{black}{f(x)=}\sum_{i=0}^{n-1}c_ix^{\textcolor{black}{p}^i}, \,\,\,\,\,\, c_i\in L.$$
The sum of a linear function and a constant is called an affine
function.

Two functions $f$ and $g$  are called {\it extended affine equivalent}
\textcolor{black}{(EA-equivalent)},
 if
$f=A_1\circ g \circ A_2+A$\textcolor{black}{,} where $A_1$ \textcolor{black}{and} $A_2$ are \textcolor{black}{linear maps
and $A$ is a constant function}.
They are called {CCZ-equivalence},
if the graph of $f$ can be obtained
from the graph of $g$ by an affine permutation. \textcolor{black}{EA-equivalence} is a
particular case of \textcolor{black}{CCZ-equivalence}; two \textcolor{black}{CCZ-equivalent} functions
preserve the APN property (for more details see \cite{Carlet:DCC:98}). In
general,  \textcolor{black}{CCZ-equivalence} is very difficult to establish.

Until 2006, the
list of known affine inequivalent APN functions on $L= \mathbb{F}_{2^n}$ was
rather short; the list consisted only of monomial functions of the
form  $f(x)=x^t$\textcolor{black}{,} for  positive \textcolor{black}{integer} $t$. In February 2006,
\textcolor{black}{Y. Edel}, \textcolor{black}{G. Kyureghyan} and \textcolor{black}{A. Pott} \cite{Edel:Kyureghyan:Poot:IEEE:06}  established (by an exhaustive search) the  first
example of an APN function not equivalent to any of
the known monomial APN  functions. Their  example is
$$x^3+ux^{36} \in \mathbb{F}_{2^{10}}[x],$$
where $u\in w \mathbb{F}_{2^5}^* \cup w^2\mathbb{F}_{2^5}^*$ and $w$ has order 3. It  is APN on $\mathbb{F}_{2^{10}}$.
Since \textcolor{black}{then} several APN polynomials have shown to be APN and not CCZ-equivalent to known power functions (by Felke, Leander, Bracken,  Budaghyan, Byrne, Markin, McGuire, Dillon and others).
This example, and other results 
opened up the possibility that there are perhaps  other sequences of APN function than the monomial APN functions  that have Gold and Kasami exponents.

\begin{defn}\label{def2}
Let $L=\mathbb{F}_q$, \textcolor{black}{with} $q=p^n$ for  positive integer
$n$. A function \\ $f:L\rightarrow L$ is called \textbf{exceptional
APN} if $f$ is APN on $L$ and also on
infinitely many extensions of $L$.
\end{defn}

From now on, we will assume $p=2$.   The main conjecture on APN functions,
formulated by Aubry, McGuire and Rodier  \textcolor{black}{\cite{Aubry:McGuire:Rodier:CM:10}} is   the following:

\textbf{CONJECTURE:} Up to equivalence, the Gold and Kasami-Welch
functions are the only exceptional APN functions.

Rodier provided a characterization for APN functions using varieties (see \cite{Rodier:CM:09}).

Let $L=\mathbb{F}_q$, with $q=2^n$. A function $f:L\rightarrow L$
is APN if and only if the affine variety $X$ with equation
$$f(x)+f(y)+f(z)+f(x+y+z)=0$$
has all its rational points contained in the surface
$(x+y)(x+z)(y+z)=0.$

Using this characterization, Rodier provided the following criteria about
exceptional APN functions.
(The bound results of Lang-Weil and
Ghorpade-Lachaud about rational points on a surface are crucial for the
proof of the next Theorem)

\begin{thm} Let $f:L\rightarrow L$, $L=\mathbb{F}_{2^n}$, a polynomial
function of degree $d$. Suppose that the variety $X$ of affine
equation
$$\frac{f(x)+f(y)+f(z)+f(x+y+z)}{(x+y)(x+z)(y+z)}=0$$
is absolutely irreducible (or has an absolutely irreducible
component over $L$), then $f$ is not an exceptional APN function.
\end{thm}

As can be seen in this theorem, proving absolute irreducibility, or the existence
of an absolutely irreducible factor, does guarantee not to be exceptional APN. In section 3 we study $X$ related
to the Kasami-Welch number (the Kasami-Welch variety $X$), its decomposition into absolutely irreducible components,
as well as some properties of these components. One of our main results is that these components
intersect transversally at a particular point. This result provides two new infinite families of absolutely
irreducible polynomials in section 4 and, in section 5, as a direct application, its contribution to the
conjecture of exceptional functions.

For the rest of the article, let us denote
\begin{equation}
\phi(x,y,z)=\frac{f(x)+f(y)+f(z)+f(x+y+z)}{(x+y)(x+z)(y+z)},
\end{equation}

\begin{equation}\label{eqn:phij:Delgado-Janwa:JA:16}
 \phi_j(x,y,z)=\frac{x^j+y^j+z^j+(x+y+z)^j}{(x+y)(x+z)(y+z)}
\end{equation}

Thus, if $f(x)=x^d+a_{d-1}x^{d-1}+a_{d-2}x^{d-2}+...+a_0$, then

\begin{equation}\label{eqn:expansion-phij:Delgado-Janwa:JA:16}
\phi(x,y,z)=\sum_{j=3}^d a_j\phi_j(x,y,z)
\end{equation}

Let us call $\phi(x,y)=\phi(x,y,1)$, $\phi_j(x,y)=\phi_j(x,y,1)$, its affine parts.

\section{The Kasami-Welch case of the conjecture of APN functions}
The well known conjecture \cite{Aubry:McGuire:Rodier:CM:10} about exceptional APN functions is:

\textbf{CONJECTURE:} Up to equivalence, the Gold and Kasami-Welch
monomial functions, $f(x)=x^{2^k+1}$ and $f(x)=x^{2^{2k}-2^k+1}$ respectively, are the only exceptional APN functions.

The conjecture was settled for monomials by Hernando and McGuire in \cite{Hernando:McGuire:JA:11}. Important results supporting this conjecture have been obtained in the last years
for polynomials. In \cite{Aubry:McGuire:Rodier:CM:10}, Aubry, McGuire and Rodier
settled the conjecture for odd degree polynomials (excluding the Gold and Kasami-Welch degree). Very recently, many important
results has been also obtained for Gold degree polynomials by Delgado and Janwa in \cite{Delgado:Janwa:ArXiv:12,Delgado:Janwa:ArXiv:16,Delgado:Janwa:DCC:16}. The Kasami-Welch case has been hardly studied up to now.
Ferard, Oyono and Rodier in \cite{Ferard:Oyono:CM:12} obtained the following results (the only two established for this case).

\begin{thm}
Suppose that $f(x)=x^{2^{2k}-2^k+1}+g(x) \in L[x]$
where $\deg(g) \leq 2^{2k-1}-2^{k-1}+1$. Let
$g(x)=\sum_{j=0}^{2^{2k-1}-2^{k-1}+1}a_jx^j$. Suppose moreover that
there exist a nonzero coefficient $a_j$ of $g(x)$ such that
$\phi_j(x,y,z)$ is absolutely irreducible. Then $f$ is not exceptional APN.
\end{thm}

For the case when the degree of $g$ is $2^{2k-1}-2^{k-1}+2$, they obtained:

\begin{thm}
Suppose that $f(x)=x^{2^{2k}-2^k+1}+g(x) \in L[x]$
where $\deg(g) \leq 2^{2k-1}-2^{k-1}+2$. Let $k \geq 3$ be odd and relatively prime to $n$.
If $g(x)$ does not have the form $ax^{2^{2k-1}-2^{k-1}+2}+a^2x^3$ then $\phi$ is absolutely irreducible, while if $g(x)$
does have this form then either $\phi$ is irreducible or $\phi$ splits into two absolutely irreducible factors which are
both defined over $L$
\end{thm}

Unlike the Gold case, for the Kasami case, it  has been  difficult to push up on the degree of $g(x)$ to obtain new results.
Part of this difficulty is due to the scarce knowledge on the decomposition of the Kasami-Welch variety.

\section{Transversal intersection on the Kasami-Welch variety}
Let $X$ the Kasami-Welch variety as defined in section 1.
The following fact, due to Janwa and Wilson \cite{Janwa:Wilson:LNCS:93}, is known.

If $t=2^{2k}-2^k+1$, then
\begin{equation}\label{eqn:kasami-products :Delgado-Janwa:JA:16}
\phi_t(x,y)=\prod_{\alpha \in \mathbb{F}_{2^k}-\mathbb{F}_2}P_{\alpha}(x,y)
\end{equation}
where $P_{\alpha}(x,y)$ is absolutely irreducible of degree
$2^k+1$ over $\mathbb{F}_{2^k}$. Furthermore, $P_{\alpha}$ satisfies $P_{\alpha}(x,0)=(x+\alpha)^{2^k+1}$.

Let $p=(1,1)$, a singular point of $X$ (Janwa and Wilson in \cite{Janwa:Wilson:LNCS:93,Janwa:McGuire:Wilson:AJA:95}, and
consequently Hernando and McGuire in \cite{Hernando:McGuire:JA:11}, classified the singularities of $X$).
One of our main results is  that $p \in P_{\alpha}$, for all $\alpha \in \mathbb{F}_{2^k}-\mathbb{F}_2$.

Let
\begin{equation}
\label{eqn:kasami-productsQ:Delgado-Janwa:JA:16}\phi_t(x,y)=\prod_{\alpha \in [\alpha]}Q_{\alpha}(x,y)
\end{equation}
where $[\alpha]$ is the conjugate class of $\alpha$ under the action of the Frobenius automorphisms (counting in the product only one representative for each conjugacy
class); \textcolor{black}{i.e.,} $Q_{\alpha} \in \mathbb{F}_2[x,y]$ is the product of
conjugate absolutely irreducible factors $P_{\alpha}$. \textcolor{black}{Since $Q_{\alpha}(p)=0$ implies that $Q_{\beta}(p)=0$ for all $\beta \in [\alpha]$},
then the fact that $p \in P_{\alpha}$, for all $\alpha \in \mathbb{F}_{2^k}-\mathbb{F}_2$, is proved if we prove that
$p \in Q_{\alpha}$, for all $\alpha \in \mathbb{F}_{2^k}-\mathbb{F}_2$.

Grouping equation  (\ref{eqn:kasami-productsQ:Delgado-Janwa:JA:16} ) in symmetric factors:
\begin{equation}\label{eqn:kasami-symmetric-products :Delgado-Janwa:JA:16}
\phi_t(x,y)=\prod_{\alpha \in [\alpha]}Q'_{\alpha}(x,y)
\end{equation}
where  $Q'_{\alpha}=Q_{\alpha}$ \textcolor{black}{ if $Q_{\alpha}$ is symmetric} and $Q'_{\alpha}=Q_{\alpha}\overline{Q_{\alpha}}$
if $Q_{\alpha}$ is not symmetric (the overline means its symmetric pair). \textcolor{black}{Thus it is enough to prove that $p \in Q'_{\alpha}(x,y)$},
\textcolor{black}{since $Q_{\alpha}(p)=0$ implies $\overline{Q_{\alpha}}(p)=0$. }
Because the aim is to prove that  $p \in P_{\alpha}$, for all $\alpha \in \mathbb{F}_{2^k}-\mathbb{F}_2$,
that is $p \in Q_{\alpha}$ for all $\alpha \in \mathbb{F}_{2^k}-\mathbb{F}_2$ (as $Q_{\alpha}$ is a relabeling of the set $P_{\alpha}$.)

\begin{lem}
\label{lem:oddED:Delgado-Janwa:JA:16}
$Q'_{\alpha}(x,y)$ has an odd number of terms of the form $x^my^m$ \textcolor{black}{for $m\geq 1$. We will refer them as "equal degree" (ED) terms.}
\end{lem}
\begin{proof}
\textcolor{black}{Since $\phi_t(0,0)=1$ and $\phi_t(x,y)$ is symmetric, by manipulation in (2),}
\begin{equation}
\textcolor{black}{\phi_t(x,y)=F_{2^{2k}-2^k-2}(x,y)+\cdots+ F_i+\cdots+F_1(x,y)+1, }
\end{equation}

\noindent where $F_i(x,y)$ is 0 or a symmetric
homogeneous \textcolor{black}{polynomial} of degree $i$.

By the symmetry of $\phi_t$, the number of non-ED terms of the form $x^m, y^n$ or $x^ny^m$ \textcolor{black}{(where $m,n > 0)$, }
will occur in pairs in $\phi_t(x,y)$. Since $p \in X$, the multiplicity of $p$, $m_p$, is a positive number. \textcolor{black}{Then,} \textcolor{black}{ $\phi_t(x+1,y+1)= G_{2^{2k}-2^k-2}(x,y)+\cdots+ G_i(x,y)+\cdots+G_{m_p}(x,y)$, }
\textcolor{black}{where $G_i(x,y)$ is 0 or a symmetric  homogeneous polynomial of degree $i$.}

\textcolor{black}{{\bf CLAIM:}    $\phi_t(x,y)$ should have an odd number of ED terms.}

  \textcolor{black}{To prove this, we first observe that since each term \textcolor{black}{ (ED or non-ED)} of $\phi_t(x,y)$ produce a constant
 term 1 in the expansion of}   $\phi_t(x+1,y+1)$,
  \textcolor{black}{because for $m,n > 0 $, $(x+1)^m, (y+1)^m$, $(x+1)^m (y+1)^n$, and $(x+1)^n (y+1)^m$,  all equal 1, when $(x,y)=(0,0)$. Now in the
 term by term expansion of $\phi_t(x+1,y+1)$, the non-constant non-ED terms , since they occur in pairs,     contribute $0 \pmod 2$, to the constant term. Therefore, the ED terms would have to contribute
 $1\pmod 2$ in this expansion of $\phi(x+1,y+1)$ for the constant term to vanish. Therefore $\phi_t(x,y)$ must have an odd number of duplicate ED terms. }

 \textcolor{black}{{\bf Proof of the Lemma}: Now,  since $\phi_t(x,y)$ has an odd number of ED-terms, and since it  is  the product of $Q'_{\alpha}$, we show that these facts  force each $Q'_{\alpha}$ to have an odd number of ED-terms. Suppose that $\phi_t(x,y)$  were a product of
two distinct  symmetric polynomials $A_{\alpha}$ and  $A_{\beta}$, with the number of ED-terms $t_\alpha$ and $t_\beta$, then  in the product  $A_{\alpha} (x,y) A_{\beta}(x,y)$, we  would get $t_\alpha t_\beta$ ED-terms.
Since $A_{\alpha}$ and  $A_{\beta}$ are symmetric, and since each non-constant non-ED term in each $A_{\alpha}$ and  $A_{\beta}$ occur in pairs, we get a contribution of an even number of ED-terms,
 say $2 t_{\alpha\beta}$. Since , $\phi_t(0,0)=1$, we have  $A_{\alpha}(0,0)=A_{\beta}(0,0)=1$, and
 in the product, the constant terms contribute $t_\alpha+t_\beta$ number of \textcolor{black}{ED} terms. Therefore,
 the number of ED-terms in the product is   $t_\alpha t_\beta+t_\alpha+t_\beta+2*t_{\alpha\beta}$, and this number is odd if and only if each  $t_{\alpha}$ and $t_{\beta}$ is odd.
 We have shown that the property of odd number of ED-terms propagate to  its two distinct factors  \textcolor{black}{that} are symmetric and have constant term 1. Since each $Q'_{\alpha}(x,y)$ is symmetric and has constant term 1, we can group them in two distinct factors $A_{\alpha}$ and $A_{\beta}$, and we conclude by induction that each $Q'(x,y)$ has an odd number of Ed-terms.
  }

\end{proof}

Two or more curves $f_i$ are said to intersect transversally at a point $p$ if $p$ is a simple point  of each $f_i$ and
if the tangent lines to $f_i$ at $p$ are \textcolor{black}{pairwise distinct}.

The next theorem is our main result.
\begin{thm}
\label{thm:transversal:Delgado-Janwa:JA:16}
In the absolutely irreducible factorization of Equation (\ref{eqn:kasami-products :Delgado-Janwa:JA:16}), the components \textcolor{black}{$P_{\alpha}(x,y)$}, $\alpha \in \mathbb{F}_{2^k}-\mathbb{F}_2$, intercept transversally at $p=(1,1)$.
\end{thm}

\begin{proof}
By (\ref{eqn:phij:Delgado-Janwa:JA:16}), the constant term of $\phi_t$ is 1. Then, as a factor of $\phi_t$ over $\mathbb{F}_2[x,y]$, the constant term of
$Q'_{\alpha}(x,y)$ is also 1, for all $\alpha \in \mathbb{F}_{2^k}-\mathbb{F}_2$. By Lemma
\ref{lem:oddED:Delgado-Janwa:JA:16}, $Q'_{\alpha}(x,y)$ has an odd number of ED terms.
Then, by direct computation, for all $\alpha \in \mathbb{F}_{2^k}-\mathbb{F}_2$, the constant term of $Q'_{\alpha}(x+1,y+1)$ is zero. Thus, for all
$\alpha \in \mathbb{F}_{2^k}-\mathbb{F}_2$, \textcolor{black}{$m_p(Q'_{\alpha}(x,y)) \geq 1$.}

\textcolor{black}{
Therefore, for all $\alpha \in \mathbb{F}_{2^k}-\mathbb{F}_2$, $p \in Q'_{\alpha}(x,y)$, and if  $Q'_{\alpha}=Q_{\alpha}\overline{Q_{\alpha}}$, then
$p\in Q_{\alpha}$ and $p\in \overline{Q_{\alpha}}$. Therefore,  $p\in Q_{\alpha}$ for all suitable $\alpha$. Therefore, we conclude  that $p \in P_{\alpha}(x,y)$, as required.}

\textcolor{black}{In addition, as shown by Janwa, McGuire and Wilson in \cite{Janwa:McGuire:Wilson:AJA:95}}, $m_p(\phi_t)=2^k-2$ and the "tangent cone" factors into $2^k-2$ different linear factors.
\textcolor{black}{We conclude that, the components $P_{\alpha}$ intersect transversally at $p$.}
\end{proof}

\section{Two new families of absolutely irreducible Kasami-Welch degree polynomials}
As some applications of Theorem \ref{thm:transversal:Delgado-Janwa:JA:16} we have:

\begin{thm} \label{thm:3mod4:Delgado-Janwa:JA:16}
Let the Kasami-Welch degree polynomial $f(x)=x^{2^{2k}-2^k+1}+h(x) \in L[x]$,
where $d=\deg(h)\equiv 3{\pmod 4}$. Then $\phi(x,y)$ is absolutely irreducible.
\end{thm}

\begin{proof}
Supposing, by the way of contradiction (as in \cite{Delgado:Janwa:ArXiv:12,Delgado:Janwa:ArXiv:16,Delgado:Janwa:DCC:16}), that $\phi(x,y)$ factor as $\phi(x,y)=P(x,y)Q(x,y)$, where $P$ and $Q$ are non constant polynomials. Then, because of the
factorization of $\phi_{2^{2k}-2^k+1}$ as a product of absolutely irreducible different factors, the following system is obtained
(see the previous references):

\begin{equation}\label{eqn:PsQt-3mod4:Delgado-Janwa:JA:16}
P_sQ_t=\prod P_{\alpha}(x,y),\,\,\,\, \alpha \in \mathbb{F}_{2^k}-\mathbb{F}_2,
\end{equation}

\begin{equation}\label{eqn:Ps-eQt-e3mod4:Delgado-Janwa:JA:16}
P_sQ_{t-e}+P_{s-e}Q_t=a_d \phi_d(x,y)
\end{equation}
where $P=P_s+P_{s-1}+...+P_0$ , $Q=Q_t+Q_{t-1}+...+Q_0$ , $2^{2k}-2^k+1 > s \geq t >0$ and $e=2^{2k}-2^k+1-d$.

By Theorem \ref{thm:transversal:Delgado-Janwa:JA:16} , $p=(1,1) \in P_{\alpha}(x,y)$ for all $\alpha \in \mathbb{F}_{2^k}-\mathbb{F}_2$. Then,-
by the absolute irreducible factorization in (\ref{eqn:PsQt-3mod4:Delgado-Janwa:JA:16}), $p \in P_s, p \in Q_t$. Then by Equation
(\ref{eqn:Ps-eQt-e3mod4:Delgado-Janwa:JA:16}), $p \in \phi_d(x,y)$. Which is a contradiction,
 since $p$ does not belong to $\phi_d(x,y)$ as demonstrated by Janwa and Wilson in \cite{Janwa:Wilson:LNCS:93}.
\end{proof}

For the case $d \equiv 1 {\pmod 4}$, it happens that $p \in \phi_d$ and $m_p(\phi_d)=2^i-2$, where $d=2^il+1$ ($l$ an odd number),
as shown by Hernando and McGuire in \cite{Hernando:McGuire:JA:11}. The following is our result for this case.

\begin{thm}  \label{thm:5mod8:Delgado-Janwa:JA:16}
Let the Kasami-Welch degree polynomial $f(x)=x^{2^{2k}-2^k+1}+h(x) \in L[x]$,
where $d=\deg(h)\equiv 5{\pmod 8}$. If $d < 2^{2k}-3(2^k)-1$ and $(\phi_{2^{2k}-2^k+1}, \phi_d)=1$, then $\phi(x,y)$ is absolutely irreducible.
\end{thm}

\begin{proof}
Supposing as in Theorem  \ref{thm:3mod4:Delgado-Janwa:JA:16} that $\phi(x,y)$ factors as $P(x,y)Q(x,y)$ and using the same arguments used there, we get the system:
\begin{equation}\label{eqn:PsQt-5mod8:Delgado-Janwa:JA:16}
P_sQ_t=\prod P_{\alpha}(x,y),\,\,\,\, \alpha \in \mathbb{F}_{2^k}-\mathbb{F}_2,
\end{equation}
\begin{equation}  \label{eqn:Ps-eQt-e-5mod8-Ps-eQt-e:Delgado-Janwa:JA:16}
P_sQ_{t-e}+P_{s-e}Q_t=a_d \phi_d(x,y)
\end{equation}
where $s,t,e$ are in the previous theorem. Let $p=(1,1)$ and let us consider the following two cases to prove the theorem.

Let $t > 2(2^k+1)$. Then using Theorem \ref{thm:transversal:Delgado-Janwa:JA:16} and the fact that $P_{\alpha}$ are absolutely irreducible polynomials of degree $2^k+1$,
from (10) we have that $m_p(Q_t) > 2$, $m_p(P_s)> 2$. This implies also that $m_p(P_sQ_{t-e}+P_{s-e}Q_t) = m_p(\phi_d(x,y))> 2$. Contradicting that
$m_p(\phi_d)=2$ (for $d \equiv 5{\pmod 8}$, $m_p(\phi_d)=2^2-2$ ).

On the other hand, let $t \leq 2(2^k+1)$. Since $d < 2^{2k}-3(2^k)-1$, then $e > 2^{2k}-2^k+1-(2^{2k}-3(2^k)-1) > 2^{k+1}+2$ and $t < e$.
The Equation (\ref{eqn:Ps-eQt-e-5mod8-Ps-eQt-e:Delgado-Janwa:JA:16}) becomes $P_{s-e}Q_t=a_d \phi_d(x,y)$, which contradicts the relatively prime hypothesis.

\end{proof}

\section{Towards the conjecture of exceptional APN functions}
As an application of Theorems        \ref{thm:3mod4:Delgado-Janwa:JA:16} and
\ref{thm:5mod8:Delgado-Janwa:JA:16}, the following two theorems contribute substantially on the hardly studied
case of the APN function conjecture, the Kasami-Welch case.

\begin{thm}  \label{thm:3mod4:APN:Delgado-Janwa:JA:16}Let the Kasami-Welch degree polynomial $f(x)=x^{2^{2k}-2^k+1}+h(x) \in L[x]$
where $\deg(h)\equiv 3{\pmod 4}$. Then $f$ is not exceptional APN.
\end{thm}

\begin{thm}  \label{thm:5mod4:APN:Delgado-Janwa:JA:16} Let the Kasami-Welch degree polynomial $f(x)=x^{2^{2k}-2^k+1}+h(x) \in L[x]$
where $\deg(h)\equiv 5 {\pmod 8}$, $d < 2^{2k}-3(2^k)-1$. If $(\phi_{2^{2^k}-2^k+1}, \phi_d)=1$, then $f$ is not exceptional APN.
\end{thm}

\textbf{Some remarks}

\textcolor{black}{In Theorem  \ref{thm:5mod4:APN:Delgado-Janwa:JA:16} , one of the conditions for $f(x)$ not to be exceptional APN is that
$(\phi_{2^{2^k}-2^k+1},\phi_d)=1$. There are many cases
for this to happen, for example, when $\phi_d$ is absolutely irreducible (this follows from the absolutely irreducible factorization of $\phi_{2^{2k}-2^k+1}(x,y)$  as given in Equation (\ref{eqn:kasami-products :Delgado-Janwa:JA:16})).
Although is well known that for $d\equiv 5{\pmod 8}$, $\phi_d$ is not always absolutely
irreducible, there are many cases where it is. In \cite{Janwa:McGuire:Wilson:AJA:95}, Janwa and Wilson proved,
using different methods including Hensel's lemma implemented on a
computer, that $\phi_d(x,y)$ is absolutely irreducible for all $3< d<100$, provided that $d$ is not a Gold
or a Kasami-Welch number. They also showed that for infinitely many values of $d$, $\phi_d$ is non-singular, and therefore absolutely irreducible. 
Recently  F{\'e}rard, in   \cite{Ferhard:AMC:14},
established sufficient conditions for $\phi_d(x,y)$ to be absolutely irreducible, when $d \equiv 5{\pmod 8}$. Among Ferard's results, by the aid of SAGE, he showed that for these type of numbers, $\phi_d(x,y)$ is absolutely irreducible for all  $d$, 13< d< 208. 
In   \cite{Ferhard:AMC:14} an infinite family of integers $d$ is given so that $\phi_d$  is not absolutely irreducible.
(We note that Hernando and McGuire, in \cite{Hernando:McGuire:JA:11}, has first showed, by using MAGMA, that $\phi_{205}(x,y)$ is not irreducible).}


\begin{thebibliography}{1}
\bibitem{Aubry:McGuire:Rodier:CM:10} Y. Aubry, G. McGuire, F. Rodier, {\it  A few more functions that are not APN infinitely often}, Comtemporary Math. American Mathematical Society,
\textbf{518} (2010) 23--31.

\bibitem{Carlet:DCC:98} C. Carlet, P. Charpin, V. Zinoviev, {\it Codes, bent functions and
permutations suitable for DES-like cryptosystems}, Designs, Codes and
Cryptography, \textbf{15} (1998) 125-156.

\bibitem{Caullery:DCC:14} F. Caullery, {\it A new class of functions not APN infinitely often}, Des. Codes Cryptogr. 73(2),601-614 (2014)


\bibitem{Delgado:Janwa:ArXiv:12} M. Delgado, H. Janwa, {\it On The Conjecture on APN
Functions}, arXiv:1207.5528v1[cs.IT] (Jul 2012).

\bibitem{Delgado:Janwa:ArXiv:16} M. Delgado, H. Janwa, {\it Progress Towards the Conjecture on APN
Functions and Absolutely irreducible polynomials}, arXiv:1602.02576v1[Math.NT] (Jan 2016).

\bibitem{Delgado:Janwa:DCC:16} M. Delgado, H. Janwa, {\it On the conjecture on APN functions and absolute irreducibility of polynomials},
Designs, Codes and Cryptography, (2016) 1--11.

\bibitem{Delgado:Janwa:AMC:16}  M. Delgado, H. Janwa, {\it Some New Results on the  Conjecture on  Exceptional APN Functions and Absolutely Irreducible Polynomials: the Gold Case},
accepted in a special issue of Advances in Mathematics of Communications (AMC), October  2016.

\bibitem{Edel:Kyureghyan:Poot:IEEE:06} Edel Y., Kyureghyan G., Pott A.: \emph{A new APN function which is not equivalent to a power mapping}.
IEEE Trans. Inf. Theory (2006).

\bibitem{Ferhard:AMC:14}
 E. F{\'e}rard, \emph{On the irreducibility of the hyperplane sections of {F}ermat
              varieties in {$\Bbb{P}^3$} in characteristic 2}, Advances in Mathematics of Communications, vol. 8, (2014),   pp. 497--509
              
              
\bibitem{Ferard:Oyono:CM:12} E. F{\'e}rard, R. Oyono, F. Rodier, \emph{Some more functions that are not {APN}
infinitely often. {T}he case of {G}old and {K}asami exponents.}, Geometry, cryptography and Coding theory.
Contemporary Mathematics, vol. 574, pp. 27--36. American Mathematical Society, providence (2012).
doi:10.1090/conm/574/11423

\bibitem{Hernando:McGuire:JA:11} F. Hernando, G. McGuire, {\it Proof of a conjecture on the sequences of exceptional numbers, classifying cyclic codes and APN functions},
Journal of algebra, \textbf{343} (2011) 78--92.

\bibitem{Janwa:McGuire:Wilson:AJA:95} H. Janwa, G. McGuire, M. Wilson, {\it Double-error-correcting cyclic
codes and absolutely irreducible polynomials over GF(2)}, Applied
Journal of Algebra, \textbf{178} (1995) 665--676.

\bibitem{Janwa:Wilson:LNCS:93} H. Janwa and M. Wilson, {\it Hyperplane sections of Fermat varieties
in $\Bbb{P}^3$ in char. 2 and some applications to cyclic codes}, Applied
Algebra, Algebraic Algorithms and Error-Correcting Codes,
Proceedings AAECC-10 (G Cohen, T. Mora and O. Moreno Eds.), Lecture
Notes in Computer Science, Springer-Verlag, NewYork/Berlin,
\textbf{673} (1993) 180--194.

\bibitem{Janwa:Wilson:IEEE:16} Janwa H., Wilson R.M.: {\it Rational points on the Klein quartic and the binary cyclic codes  $<m_1m_7>$}. IEEE Trans. Inf. Theory (to appear).

\bibitem{Jedlicka:FFA:07} D. Jedlicka, {\it APN monomials over $GF(2^n)$},
Finite Fields Appl., \textbf{13} (2007) 1006--1028.

\bibitem{Lidl:Niederreiter:C:97} Lidl R., Niederreiter H.: {\it Finite fields. Encyclopedia of Mathematics and Its Applications}, vol. 20,
2nd edn. Cambridge University Press, Cambridge (1997). With a foreword by P. M. Cohn.


\bibitem{Nyberg:Knudsen:LNCS:92} K. Nyberg, L. R. Knudsen, \emph{Provable security against differential attacks}.
Lect. Notes Comput. Sci., vol.740, Springer, pp. 566-574
 \hskip 1em plus 0.5em minus 0.4em\relax 1992.



\bibitem{Rodier:CM:09} F. Rodier, {\it Bornes sur le d\`egre des polyn\`aomes presque parfaitement
non-lin�eaires}, Contemporary Math., AMS, Providence (RI), USA,
\textbf{487} (2009) 169--181.

\bibitem{Rodier:Hal:11}  Rodier F. {\it Some more functions that are not APN
infinitely often. The case of Kasami exponents}, Hal-00559576,
version 1-25 (Jan 2011).

\bibitem{Rodier:CC:16} F. Rodier, \emph{Functions of degree {$4e$} that are not {APN} infinitely often},
Cryptogr. Commun. \textbf{3}(4), 227--240, (2016)

\bibitem{Schmidt:KP:04} Schmidt W.: \emph{Equations Over Finite Fields: An Elementary Approach}, 2nd edn,
Kendrick Press, Heber City (2004).

\end{thebibliography}
\end{document}